\documentclass[12pt]{amsart}
\usepackage{amsfonts,amsmath,amssymb,amsthm,latexsym,verbatim}
\usepackage{tikz}
\usetikzlibrary{arrows, plotmarks}
\usetikzlibrary{quotes, angles, calc,intersections}
\usepackage[np,autolanguage]{numprint}
\newtheorem{theorem}{Theorem}[section]
\newtheorem{lemma}[theorem]{Lemma}

\newtheorem{corollary}[theorem]{Corollary}
\theoremstyle{remark}
\newtheorem{remark}[theorem]{Remark}
\theoremstyle{definition}

\newtheorem{example}[theorem]{Example}
\DeclareMathOperator{\Ker}{{\mathrm{ker}}}

\DeclareMathOperator{\image}{{\mathrm{Im}}}
\DeclareMathOperator{\Hom}{{\mathrm{Hom}}}

\DeclareMathOperator{\PG}{{\mathrm{PG}}}
\DeclareMathOperator{\AG}{{\mathrm{AG}}}
\newcommand{\abs}[1]{|#1|}	
\newcommand{\Fq}{{\mathbb {F}_q}} 
\newcommand{\Fp}{{\mathbb {F}_p}} 
 
\newcommand{\K}{{\mathcal {K}}} 
\newcommand{\R}{{\mathbb {R}}} 
 
\newcommand{\Q}{{\mathbb {Q}}} 

\newcommand{\CC}{{\mathcal {C}}} 
\newcommand{\Trace}{{\mathrm {Tr}}} 
\DeclareMathOperator{\rank}{{\mathrm {rank}}} 	

\begin{document}
\title[]{Linear representations of finite geometries and associated LDPC codes.}

\author[Sin, Sorci and Xiang]{Peter Sin, Julien Sorci and Qing Xiang}
\address{Peter Sin, Department of Mathematics, University of Florida, P. O. Box 118105, Gainesville FL 32611, USA}\email{sin@ufl.edu}
\address{Julien Sorci, Department of Mathematics, University of Florida, P. O. Box 118105, Gainesville FL 32611, USA}\email{jsorci@ufl.edu}
\address{Qing Xiang, Department of Mathematical Sciences, University of Delaware, Newark, DE 19716, USA} \email{qxiang@udel.edu}
\thanks{The third author was partially supported by NSF grant DMS-1600850}
\date{}


\begin{abstract}
The {\it linear representation} of a subset of a finite projective space
is an incidence system of affine points and lines determined by the subset.
In this paper we use character theory to show that the rank of the incidence matrix has a direct geometric interpretation in terms of certain hyperplanes.  
We consider the LDPC codes defined by taking the incidence matrix and its transpose as  parity-check matrices, and in the former case prove a conjecture of Vandendriessche that the code is generated by words of minimum weight called plane words. In the latter case we compute the minimum weight in several cases and provide explicit constructions of minimum weight codewords. 
\end{abstract}

\maketitle
\section{Introduction} 
Codes with sparse parity-check matrix were first considered in the seminal PhD dissertation of Gallagher \cite{RGG}. Such a code is called a low-density parity-check, or LDPC, code. While Gallagher's ideas were largely overlooked for many years, since the 1990s there has been a resurgence of interest in constructing LDPC codes due to their relatively fast decoding algorithms while still achieving high rates of transmission. One method of constructing these codes is by taking the incidence matrix of a finite geometry as parity-check matrix, which is the method that we shall consider here.

Let $q= p^e$ be a prime power and
$n\ge 2$ be an integer. Let $E$ be an $(n+1)$-dimensional vector space over $\Fq$ and $V$ an $n$-dimensional subspace. Then $H:=\PG(V)\cong\PG(n-1,q)$ is a hyperplane of
$\PG(E)\cong\PG(n,q)$. Let $P:=\PG(E)\setminus H \cong \AG(n,q)$ be the complementary affine space, which we shall often view as a set of $q^n$ affine points with $H$ as the hyperplane at infinity. Additionally, if $\ell$ is an affine line of $P$ with point at infinity $u$ in $H$, we will at times refer to $u$ as the direction of the line $\ell$, and thus view the points of $H$ as the directions of affine lines in $P$. 

Fixing an arbitrary subset $\K$ of $H$, we consider the point-line incidence system whose point set is $P$ and whose line set $L$ is the set of affine lines of $P$ whose direction is in $\K$. Each point of $\K$ is the direction of $q^{n-1}$ parallel lines in $L$, so that $\abs{L}=q^{n-1}\abs{\K}$.
A line in $L$ will be viewed as a set of $q$ points of $P$ and we define
a line to be incident to its points.
This point-line incidence system is called the {\it linear representation}
of the set $\K$ and is denoted by $T_{n-1}^*(\K)$ (Figure~\ref{T*}). It has been studied for many choices of $\K$ 
in \cite{dCvM}, and the automorphism group has been studied in \cite{CRV2}.
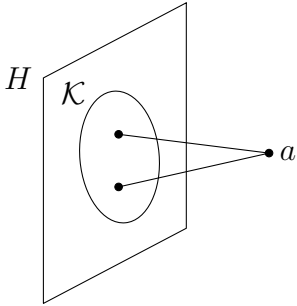
\begin{figure}\label{T*}
\begin{tikzpicture}[scale=1]
	\draw (4,1) -- (4,4) -- (5.9,5) -- (5.9,2) -- (4,1);
	\draw[rotate=5] (5.25,2.5) ellipse (15pt and 25pt);
	\draw[fill] (5,3.25) circle [radius=0.05]; 
	\draw[fill] (5,2.55) circle [radius=0.05];
	\draw[fill] (7,3) circle [radius=0.05];
	\draw (5,3.25) -- (7,3) -- (5,2.55);
	\node at (4.4,3.8) {$\mathcal K$};
	\node[left] at (4,4) {$H$};
	\node[right] at (7,3) {$a$};
\end{tikzpicture}
\caption{A point $a$ and two lines in the geometry $T_{n-1}^*(\K)$.}
\end{figure}

We introduce homogeneous coordinates $X_0$, $X_1$,..,$X_n$ in $\PG(n,q)$ and without loss of generality let $H$ be the hyperplane $X_0=0$. We fix arbitrary orderings on $P$ and $L$ and let $N$ denote the point-line incidence matrix.  
We have defined $N$ as an integer matrix, but we can consider it as
a matrix over any field.  The following theorem gives a geometric interpretation of the rank of $N$.

\begin{theorem}\label{main} Let $F$ be a field in which $q\neq 0$.
Then $\rank_F N$ is  equal to the number of functions  in 
the dual space $V^*= \textnormal{Hom}(V,\Fq)$ that take the value zero at some point of $\K$.
In geometric terms, $\rank_F N=1+(q-1)h_\K$, where $h_\K$ is the number of hyperplanes in
the projective space $H$ that have nonempty intersection with $\K$.
\end{theorem}

\begin{remark}\label{rem1}
If $q=0$ in $F$, then $\rank_FN$ will be bounded above
by $\rank_\Q N$, which is given by the theorem. However, small examples
show that inequality will be strict in general, and we have no conjecture
yet for the  exact value of $\rank_FN$ in this case. 
\end{remark}

We define the $F$-codes $\mathcal C$ and $\mathcal D$ to be the codes with parity-check matrices $N$ and $N^T$, respectively.  Each row of $N$ has weight $|\mathcal K|$ and each column has weight $q$ so that $\mathcal C$ and $\mathcal D$ are LDPC codes of lengths $|\mathcal K|q^{n-1}$ and $q^n$, respectively, with dimensions given by Theorem~\ref{main}. The code $\mathcal C$ has been studied in \cite{V}, \cite{PSV} and \cite{V2}. In \cite{KLF}, Kou et al. considered the codes $\mathcal C$ and $\mathcal D$ for $F=\mathbb{F}_2$ and when $\mathcal K$ is the set of all points of $H$, referring to $\mathcal C$ as the type-II geometry-\textbf{G} LDPC code and $\mathcal D$ as the type-I geometry-\textbf{G} LDPC code. In this case the geometry $T_{n-1}^*(\mathcal K)$ is isomorphic to $\textnormal{AG}(n,q)$, so the parity-check matrix for $\mathcal C$ is the full incidence matrix of $\textnormal{AG}(n,q)$ and the parity-check matrix for $\mathcal D$ is the transpose. Kou et al. noted that $\mathcal C$ is one-step majority-logic decodable and can correct $\lfloor q/2 \rfloor$ errors with this decoding scheme, so that the minimum distance of $\mathcal C$ is at least $q+1$. Similarly, they also noted that using the same decoding algorithm, one can correct $\lfloor |\mathcal K|/2 \rfloor$ errors in $\mathcal D$ so that $\mathcal D$ has minimum distance at least $|\mathcal K|+1$.   

In \cite{TXK}, Tang et al. studied the codes $\mathcal C$ and $\mathcal D$ when $F= \mathbb{F}_2$ and for an arbitrary subset of points $\mathcal K$ in $H$, referring to $\mathcal C$ as the \textit{Complementary Gallager Euclidean Geometry LDPC code} and $\mathcal D$ as the \textit{Gallager Euclidean Geometry LDPC code}. For these more general codes, they also provided the lower bounds $q+1$ and $|\mathcal K|+1$ on the minimum distances of $\mathcal C$ and $\mathcal D$, respectively, by observing that any $q$ columns of $N$ or $|\mathcal K|$ columns of $N^T$ are linearly independent. We note here that these lower bounds  hold for any field $F$, and any prime power $q$ by the following argument. A codeword of $\mathcal D$ is an element of $F^P$, whose support is a set of points in $P$. If $x$ is a codeword of $\mathcal D$, and $s$ is a point in the support of $x$, then each of the $|\mathcal K|$ lines meeting $s$ must contain a further point of the support. These further points are distinct, hence the weight of $x$ is at least $|\mathcal K|+1$. Similarly, any codeword $x$ in $\mathcal C$ has the property that if $\ell$ is a line in the support of $x$, then for each of the $q$ points incident to $\ell$ there is a further line of the support meeting $\ell$ at that point, so that $x$ must have weight at least $q+1$.

The proof of Theorem~\ref{main} will be given in the next section, but we can
sketch the main ideas now. We consider the actions of the additive group of the
vector space $V$ on the affine space $P$ and on the set $L$ of lines. These actions turn the spaces $F^P$ and $F^L$ of $F$-valued functions on $P$ and $L$ into $FV$-modules, where $FV$ is the group algebra of $V$, and the incidence relation then defines
an  $FV$-module homomorphism  $F^L\to F^P$, whose matrix is $N$.
The regular action of $V$ on $P$ gives in addition an $FV$-isomorphism
from $F^P$ to the space $F^V$ of $F$-valued functions on $V$.
Thus, $\rank_FN$ equals the dimension of the image in $F^V$ of the
composite homomorphism $F^L\to F^V$. The standard bases for  $F^L$, $F^P$ and $F^V$
consist of the characteristic functions of elements. When $F$ contains a primitive $p$-th root of unity, $F^V$ has a second natural basis, namely the group $\hat V=\Hom(V, F^\times)$ of $F$-characters of $V$. Using character theory, we can reduce the rank problem
to one of counting certain characters. Finally, we make use of  a natural bijection between
$\hat V$ and the dual space $V^*=\Hom_\Fq(V,\Fq)$  to arrive at Theorem~\ref{main}.

In \S\ref{planewords} we consider the $F$-codes $\mathcal C$ and $\mathcal D$ and prove a conjecture of P. Vandendriessche \cite{V2}
that the code $\mathcal C$ is generated by certain codewords called {\it plane words}, which are
known to be words of minimum weight. We will also define words in $\mathcal D$ analogous to plane words called \textit{capacitor words} and  show that they are codewords that span $\mathcal D$. 
In \S\ref{app} we apply Theorem~\ref{main} in particular cases of $n$ and $\K$.
In the case where $\K$ is a rational normal curve minus its point at infinity
the sets $P$ and $L$  form the bipartition of the vertex set of a bipartite
graph called the Wenger graph. The Wenger graphs have been studied extensively;
their automorphism groups have been found in \cite{CRV} and their spectra
determined in \cite{CLL}. Our theorem shows that the
rank of the adjacency matrix of a Wenger graph 
over any field $F$ in which $q\neq0$,
is the same as the real rank, hence equal to the matrix size minus the
multiplicity of zero as an eigenvalue of the adjacency matrix. The 
multiplicity of every eigenvalue has been computed in \cite{CLL}.
However we shall derive the rank independently, directly from Theorem~\ref{main}. In the special case $n=3$, we obtain a simple new proof for the dimensions
of the LDPC codes called $\mathrm{LU}(3,q)$, that were first computed in \cite{SX}
for fields of characteristic $2$, and for other fields with $q\neq0$ in \cite{V}.

\section{Proof of Theorem~\ref{main}}\label{proof}
In proving Theorem~\ref{main} $F$ may be replaced by any extension field $F'$,
as  $\rank_{F'}N=\rank_FN$, so we assume from now on that $F$ is a field in which $q\neq 0$, that
contains a primitive $p$-th root of unity $\omega$. Let $F^L$ be the space of $F$-valued functions on $L$ and $F^P$ the space of $F$-valued functions on $P$. These vector spaces are actually $FV$-permutation modules, as $V$
permutes the sets $L$ and $P$ by an action that we now describe. A point of $P$ has homogeneous coordinates $(1:a_1:\cdots:a_n)$ with $a_i\in\Fq$ for all $i$, and together with $(0:b_1:\cdots:b_n)\in \K$ determines an affine line consisting of the $q$ points $(1:a_1+tb_1:\cdots:a_n+tb_n)$ for $t\in \Fq$.
The additive group of $V$ acts regularly on $P$ via the action
\[ (0,v_1,\dots,v_n) \cdot (1:a_1:\cdots:a_n) = (1:a_1+v_1:\cdots:a_n+v_n)   \] 
From the description of lines in $L$ above it follows that the
action of the group $V$ on $P$ induces an incidence-preserving action on $L$, where the vector $(0,v_1,\dots,v_n)\in V$  moves the line $\{(1:a_1+tb_1:\cdots:a_n+tb_n)\mid t\in \Fq\}$ to the line $\{(1:a_1+v_1+tb_1:\cdots:a_n+v_n+tb_n)\mid t\in \Fq\}$.

There are natural bases for the vectors spaces $F^L$ and $F^P$: For $\ell \in L$ we denote by $[\ell]\in F^L$ the characteristic function of the element $\ell$, which takes value one at $\ell$ and zero elsewhere. These form a basis of $F^L$. Similarly, for $a\in P$, we let $\chi_a\in F^P$ denote the
characteristic function of $a$, and the $\chi_a$ form a basis of $F^P$. More generally, for any subset $A$ of $P$ we let $\chi_A$ denote the characteristic function of $A$.  We consider the incidence map $\eta:F^L\to F^P$, defined by $[\ell] \mapsto \chi_\ell=\sum_{a\in \ell}\chi_a$. With respect to the bases $\{[\ell]\}_{\ell \in L}$ of $F^L$ and $\{\chi_a\}_{a\in P}$ of $F^P$, the matrix of $\eta$ is $N$, considered as a matrix over $F$.

The set  $P$ is isomorphic, as a $V$-set, to the set $V$, with the left regular action, under the map taking each point $(1:a_1:a_2:\cdots:a_n)$ to the vector  $(0,a_1,a_2,\ldots,a_n)$. Hence we have an induced $FV$-module isomorphism $\sigma:F^P\to F^V$, 
mapping the characteristic function of a point to the characteristic function of the corresponding vector. 
We shall study the composite map $\sigma\eta:F^L\to F^V$, since the dimension of its image
is equal to $\rank_FN$.  The basis element $[\ell]$ for $\ell \in L$ is mapped to the sum of all the  characteristic functions of vectors corresponding to points of $\ell$.

Next, we briefly present the character theory necessary for the proof of Theorem~\ref{main}. For $\theta\in V^*=\Hom_\Fq(V,\Fq)$, the dual vector space to $V$,  we can define an $F$-character $\lambda$ by $\lambda(v)=\omega^{\Trace(\theta(v))}$, for all $v\in V$, where $\Trace$ is the trace map from $\Fq$ to $\Fp$.
In this way we obtain a bijection from $V^*$ to the group
$\hat V:=\Hom(V,F^\times)$ of $F$-characters of $V$.
We shall denote the character corresponding to the linear function $\theta$ by  $\lambda_\theta$ and the linear function corresponding to a character $\lambda$ by $\theta_\lambda$.
An element of $V^*$ can be represented in the usual way using dual coordinates as 
$\theta=[c_1,c_2,\ldots, c_n]$, where $\theta(0,v_1,\ldots, v_n)=c_1v_1+c_2v_2+\cdots+c_nv_n$  for $(0,v_1,v_2,\ldots, v_n)\in V$. 
Note that if $W$ is an $\Fq$-subspace of $V$, then for $\theta\in V^*$ we have
\begin{equation}\label{trace}
\Trace(\theta(W))=\begin{cases} \Fp,\quad\text{if $W\nsubseteq\Ker(\theta)$},\\ 
{\{0\}},\quad\text{if $W\subseteq\Ker(\theta)$.}
\end{cases}
\end{equation}

\begin{lemma}\label{kertheta}
If $W$ is an $\Fq$-subspace of $V$ and $\lambda\in\hat V$, then
$\sum_{w\in W}\lambda(w)\neq 0$ if and only if $W\subseteq\Ker(\theta_\lambda)$.
\end{lemma}
\begin{proof} Suppose $W\nsubseteq\Ker(\theta_\lambda)$. Then by \eqref{trace}
the composite map $\Trace \circ \theta_\lambda$ is a surjective group homomorphism
from $W$ to $\Fp$. Therefore, $\sum_{w\in W}\lambda(w)$ is a multiple of
the complete sum of $p$-th roots of unity, hence equal to zero. Conversely,
if $W\subseteq\Ker(\theta_\lambda)$, then $\sum_{w\in W}\lambda(w)=\abs{W}\neq 0$.
\end{proof}

Let $F^V$ be the space of $F$-valued functions on $V$, with $V$ acting
by the formula $(vg)(w)=g(w-v)$, for $v$, $w\in V$ and $g\in F^V$.
The space $F^V$ has two bases. The first is the set of characteristic functions
$\delta_v$ of elements $v\in V$; the second is $\hat V$. 
By the orthogonality relations we have, for $v\in V$,
\begin{equation}\label{delta}
\delta_v=\frac{1}{\abs{V}}\sum_{\lambda\in\hat V}\lambda(-v)\lambda
\end{equation}

For $g\in F^V$, we write $g=\sum_{\lambda\in\hat V}a_\lambda\lambda$ as
an $F$-linear combination of characters, where $a_\lambda = \frac{1}{|V|} \sum_{v \in V} g(v) \lambda(-v)$; if $a_\lambda =0$ then we say that $\lambda$ is a root of $g$; otherwise $\lambda$ is a nonroot of $g$. Define
$\hat V_g:=\{\lambda\in\hat V\mid\, a_\lambda\neq 0\}$, the set of nonroots of $g$.

The following is a special case of a well-known general principle. 
\begin{lemma}\label{FVg} For each $g\in F^V$, the set $\hat V_g$ is a basis for the
$FV$-submodule of $F^V$ generated by $g$.
\end{lemma}
\begin{proof}It is clear that the $F$-span of $\hat V_g$ contains $g$. Since
the span is itself an $FV$-submodule, it contains the $F$-submodule $FVg$ 
generated by $g$. Conversely, for each $\lambda\in\hat V$,
$FV$ contains the idempotent $e_\lambda=\frac{1}{\abs{V}}\sum_{v\in V}\lambda(v)v$
and an easy computation using the orthogonality relations shows that $e_\lambda$ acts as the identity on $\lambda$ while annihilating all other characters. Thus $e_\lambda g=a_\lambda\lambda$, and so $FVg$ contains $\hat V_g$.
\end{proof}

\begin{remark}
	In coding theory terms, for each $g \in F^V$, the $F$-submodule $FVg$ is the ideal $F[V]$ generated by $g$. The lemma is saying that the dimension of $FVg$ is the number of nonroots of $g$. 
\end{remark}

We are now ready to complete the proof of Theorem~\ref{main}. If $\ell=\{(1:a_1+tb_1:a_2+tb_2:\cdots :a_n+tb_n)\mid t\in \Fq\}$, then 
\begin{equation}\label{line_image}
\sigma\eta([\ell])=\sum_{t\in\Fq}\delta_{(0,a_1+tb_1,a_2+tb_2,\ldots, a_n+tb_n)}.
\end{equation}
In particular, if $\ell \in L$ is a line through $(1:0:0:\cdots:0)\in P$ then
the elements of $V$ corresponding to points of $\ell$ 
form the one-dimensional subspace $W_\ell$ of $V$ generated by $(0,b_1,b_2\cdots,b_n)$.
Thus,
\begin{equation}
\begin{aligned}
\sigma\eta([\ell])&=\sum_{w\in W_\ell}\delta_{w}\\
&=\frac{1}{\abs{V}}\sum_{w\in W_\ell}\sum_{\lambda\in\hat V}\lambda(-w)\lambda\\
&=\frac{1}{\abs{V}}\sum_{\theta\in V^*}[\sum_{w\in W_\ell}\lambda_\theta(w)]\lambda_\theta,
\end{aligned}
\end{equation}
since $\hat V$ is parametrized by $V^*$ and negation permutes $W_\ell$. By Lemma~\ref{kertheta}, the coefficient of $\lambda_\theta$ is nonzero if and only if
$W_\ell\subseteq\Ker(\theta)$, or in other words,
$$
\hat V_{\sigma\eta([\ell])}=\{\lambda_\theta\in \hat V \mid\, W_\ell\subseteq\Ker(\theta)\}.
$$ 

Let $L_0\subset L$ denote the set of lines through $(1:0:0:\cdots:0)\in P$.

By Lemma~\ref{FVg} the set 
\begin{equation}\label{imagebase}
\{\lambda_\theta\in \hat V \mid\ \exists \ell \in L_0,\   W_\ell \subseteq\Ker(\theta)\}.
\end{equation}
is a basis for the $FV$-submodule of $F^V$  generated by the images of lines in $L_0$. 
As every line of $L$ is in the $V$-orbit of a line through $(1:0:0:\cdots:0) \in P$, this submodule
is equal to $\image(\sigma\eta)$.
For each $\ell \in L_0$,  the subspace $W_\ell$ may be viewed as a point of $H$, 
and as such it is the point at infinity of the line $\ell$. In this way, $L_0$ 
corresponds bijectively with $\K$, so the basis  \eqref{imagebase} of $\image(\sigma\eta)$
is equal to
\begin{equation}\label{imagebase2}
\{\lambda_\theta\in \hat V \mid\ \exists u\in \K,\  \theta(u)=0\}.
\end{equation}
Thus, the first statement of Theorem~\ref{main} is proved.
For the last statement of  Theorem~\ref{main} we observe that
 nonzero linear functions in $V^*$ define hyperplanes of $H$ and 
two such functions define the same hyperplane if and only if each is a
nonzero scalar multiple of the other. Thus, if $h_\K$ denotes the
number of hyperplanes of $H$ that meet $\K$, there are 
$1+(q-1)h_\K$ linear functions in $V^*$ that take the value $0$ at some point of $\K$.

\section{The codes associated to $T_{n-1}^*(\mathcal K)$}\label{planewords}

\subsection{The Code $\mathcal C$} The code $\CC\leq F^L$ is defined as the kernel of $\eta$. That is, $\CC$ is the $F$-linear code with $N$ as its parity check matrix, hence its dimension
can be found immediately from Theorem~\ref{main}. 
We shall next define a subcode $\CC'$, following \cite{V2}. 
A {\it plane word} is defined as follows.
Let $u_1$ and $u_2$ be two points of $\K$. There are $\frac{q^{n-1}-1}{q-1}$ planes in $\PG(E)$ containing the line joining $u_1$ and $u_2$, among which there are $\frac{q^{n-2}-1}{q-1}$ planes contained in $H$. We shall refer to these $q^{n-2}$ planes not contained in $H$ as \textit{affine planes}.  Let $T$ be an affine plane of
$\PG(E)$ whose line at infinity in $H$ is the line joining $u_1$ and $u_2$. Then the 
plane word $w(u_1, u_2, T)$  is the sum of the characteristic functions of
the $q$ affine lines of $T$ having $u_1$ as point at infinity minus the 
corresponding sum with respect to $u_2$. A plane word clearly belongs to $\CC$ 
and it has weight $2q$. We denote by $\CC'$ the subcode spanned
by the plane words for all possible choices of $u_1$, $u_2 \in \mathcal K$ and $T$.

\begin{theorem}\label{codes} Assume that $q\neq 0$ in $F$. Then $\CC'=\CC$. \end{theorem}
\begin{proof} Since we will be considering geometries and codes associated with different subsets $\K$ of $H$, we will adopt
notation such as $L_\K$ for the set of lines of the geometry $T_{n-1}^*(\mathcal K)$,  $\eta_\K$ for the incidence map, and $\CC_\K$, $\CC'_\K$ for the codes.
We proceed by induction on $\abs{\K}$, the case of $\K=\emptyset$ being trivial.
Assume inductively that for some $\K$ we have $\CC_\K=\CC'_\K$ and let $u_0$ be a point of $H$ outside $\K$, and $\K'=\K\cup\{u_0\}$.
We shall consider the vector space $\overline E=E/u_0$ and its projective space $\PG(\overline E)$. We use the bar convention for images  in $\PG(\overline E)$ of
objects in $\PG(E)$. Thus $\overline {H}=\PG(\overline V)$ is a hyperplane of $\PG(\overline E)$ and $\overline {P}$ is its affine complement. The image $\overline \K$
of $\K$  in $\overline {H}$ has one point for each line of $H$ through $u_0$ that meets $\K$. Let $\overline {L}$ denote the set of affine lines in $\overline {P}$ with point at infinity in $\overline\K$. 
Under the projection from $\PG(E)$ to $\PG(\overline E)$ the set $L_{\{u_0\}}$ maps bijectively to $\overline {P}$ and the set $\mathcal T$ of affine planes $T$ with line
at infinity passing through $u_0$ and some point of $\mathcal K$ maps bijectively to $\overline {L}$.
 Naturally, these bijections preserve incidence. Thus, the incidence system $(L_{\{u_0\}},\mathcal T)$ is isomorphic to $T_{n-2}^*(\overline\K)$. 
In the decomposition
\begin{equation}
F^{L_{\K'}}=F^{L_{\{u_0\}}}\oplus F^{L_\K}
\end{equation}
let $\pi$ be the projection onto the first summand $F^{L_{\{u_0\}}}$.  Then the projection $\pi(w(u_0,u_1,T))$ of a plane word is just the sum of the characteristic functions of the $q$ affine lines of $T$ having $u_0$ as point at infinity.
Then, under the isomorphism  of $F^{L_{\{u_0\}}}$ with  $F^{\overline {P}}$ 
induced by the above bijection, these $q$ affine lines become the
$q$ points of the affine line $\overline T$, 
so  the image in  $F^{\overline {P}}$  of $\pi(w(u_0,u_1,T))$ is the characteristic function of $\overline T$. 
It follows that $\pi(\CC'_{\K'})$ is isomorphic to the subspace of $F^{\overline {P}}$ 
spanned by the characteristic functions of lines  in $\overline {L}$,
so its dimension is equal to the rank of 
$\eta_{\overline\K}:F^{\overline {L}}\to F^{\overline {P}}$.
Since $\pi(\CC'_\K)=\{0\}$ we have shown that
\begin{equation}
\dim \CC'_{\K'}-\dim\CC'_\K\geq \rank_F\eta_{\overline\K}.
\end{equation}
By our induction hypothesis and the fact that $\mathcal C_{\K^\prime}^\prime \subseteq \mathcal C_{\K^\prime}$, we therefore obtain
\begin{equation}\label{diffC}
\dim \CC_{\K'}-\dim\CC_\K\geq \rank_F\eta_{\overline\K},
\end{equation}
with equality if and only if $\CC'_{\K'}=\CC_{\K'} $.
On the other hand by considering the linear maps $\eta_\K$  and $\eta_{\K'}$ we obtain
\begin{equation}
\dim \CC_{\K'}-\dim\CC_\K=q^{n-1}-(\rank_F\eta_{\K'}-\rank_F\eta_\K).
\end{equation}
By Theorem~\ref{main}, 
\begin{equation} 
\begin{aligned}
\rank_F\eta_{\K'}-\rank_F\eta_\K&=\#\{\theta \in V^*\mid \text{ $\theta(u_0)=0$ and $\theta(u)\neq 0$ for all $u\in \K$}\}\\
&=\#\{\theta \in \overline V^*\mid \text{ $\theta(u)\neq 0$ for all $u\in \overline\K$}\}\\
&=q^{n-1}-\rank_F\eta_{\overline\K}.
\end{aligned}
\end{equation}
Therefore, we have equality in \eqref{diffC} and the inductive step is established.
\end{proof}

\begin{remark}
Theorem~\ref{codes} was conjectured in \cite{V2}, where the case $n=3$ was proved.
Now that Theorem~\ref{codes} is established, the minimum weight of $\CC$ (for arbitrary $n$ and $\K$) is given by  \cite[Theorem 5.4]{V2}. As long as $\CC\neq 0$ the plane 
words are words of minimum weight $2q$ (although not the only words of this weight in general). Thus, in all cases, $\CC$ is generated by its words of minimum weight. 
\end{remark}

\subsection{The Code $\mathcal D$}
We begin our discussion of $\mathcal D$ by describing a generating set of codewords with a geometric flavor similar to that of the plane words. First we need some expressions regarding indicator functions. If $U$ is an affine hyperplane of $P$, then we can view $U$ as a coset of a subgroup $U_0$ of $V$, say $U= U_0 + x$, where $U_0$ is the hyperplane parallel to $U$ containing $0$. We can express the indicator function of $U_0$ in terms of characters
\begin{equation}\label{chiU0}
	\begin{split}
		\chi_{U_0} = \frac{1}{q} \sum_{\lambda \in \widehat{(V/U_0)}} \lambda
	\end{split}
\end{equation}
	where $\widehat{(V/U_0)}$ denotes the characters of $\widehat{V}$ whose kernels contain $U_0$. The indicator function of the coset $U= U_0 +x$ is simply the translate of $\chi_{U_0}$ by $x$, given by
\begin{equation}
 	\begin{split}\label{chiU}
		\chi_{U} = \frac{1}{q} \sum_{\lambda \in \widehat{(V/U_0)}} \lambda(-x)\lambda.
	\end{split}
\end{equation}

Let $U_1, U_2$ be parallel affine hyperplanes of the affine space $P$
whose  common subspace at infinity is a hyperplane $T$ of $H$
that does not meet $\mathcal K$. We define a \textit{capacitor word} of $U_1, U_2$ to be the indicator function of $U_1$ minus that of $U_2$. 
\begin{theorem}
Let $F$ be a field with $q\neq 0$. The capacitor  words are codewords generating the $F$-code $\mathcal D_{\mathcal K}$.	
\end{theorem}

\begin{proof}
  Let $w$ be a capacitor word on the affine hyperplanes $U_1$ and $U_2$. The condition
  on their common subspace at infinity implies that every line $\ell \in L$ meets the hyperplanes $U_1$ and $U_2$ each at exactly one point, so that $w$ is indeed a codeword of $\mathcal D_\mathcal K$.
  To show that the capacitor words generate $\mathcal D_{\K}$, let $W$ be the $F$-subspace generated by the capacitor words on all pairs of parallel affine hyperplanes with
  subspace at infinity  not meeting $\K$. Then $W\subseteq \mathcal D_{\K}$.  By Theorem~\ref{main} the dimension of $\mathcal D_{\K}$ is equal to the number of functions in $V^*$ that do not contain any point of $\K$ in its kernel. Each such function $\theta\in V^*$ corresponds bijectively to a character $\lambda_\theta\in\widehat{V}$, as in \S2. We shall show that each such character lies in $W$, which will imply that $\dim W\geq \dim \mathcal D_{\K}$, and hence that $W=\mathcal D_{\K}$.
    Thus, let $\theta\in V^*$ be given, such that no point of $\K$ lies in
    $\Ker(\theta)$. Let $U_0=\Ker(\theta)$, an $\Fq$-hyperplane of $V$, and let
    $\lambda_\theta\in\widehat{V/U_0}$ be the corresponding character.
    Our aim is to show that  $\lambda_\theta\in W$. By Lemma~\ref{FVg}, it suffices
    to show that $\lambda_\theta$ appears with nonzero coefficient in some
  capacitor word when the latter is expressed as a linear combination of characters.
  We can choose $x\in V$ such that $\lambda_\theta(x)\neq 1$. Then $x\notin U_0$, and
  from \ref{chiU0} and \ref{chiU} we see that the indicator function of the capacitor word based on $U=U_0+x$ and $U_0$ is equal to
  \begin{equation}
   \chi_{U}-\chi_{U_0} = \frac{1}{q} \sum_{\lambda \in \widehat{(V/U_0)}}
   (\lambda(-x)-1)\lambda.
\end{equation}
  By choice of $x$, the coefficient of $\lambda_\theta$ in this expression is nonzero,
  and the proof is complete.
  \end{proof}

Similarly, we can define a codeword whose weight depends on the dimension of the subspace spanned by the points of $\mathcal K$. Suppose that the points of $\mathcal K$ span a $d$-dimensional subspace of $H$ denoted by $X$, and let $Y$ be an affine $(d+1)$-dimensional space whose intersection at infinity is $X$, and $T$ a hyperplane of $X$ not meeting $\mathcal K$. Then let $S_1$ and $S_2$ be hyperplanes of $Y$ whose intersections at infinity are both $T$. Define a \textit{d-capacitor word} as the sum of the characteristic functions of points of $S_1$ minus the corresponding sum of points of $S_2$. It is easy to see that this is indeed a codeword of $\mathcal D$ of weight $2q^d$: given any line $\ell$ of $L$, either $\ell$ has no points contained in $Y$, or $\ell$ is totally contained in $Y$. In the latter case, $\ell$ must meet each of the hyperplanes $S_1$ and $S_2$ in exactly one point, and hence we have a sum of zero when summing over the points incident to $\ell$. 

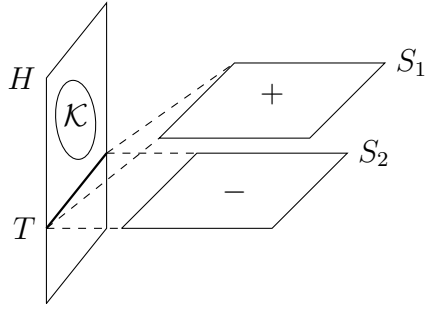
\begin{figure}
\begin{tikzpicture}[scale = 1]
	\draw (0,0) -- (0,3) -- (.8,4) -- (.8,1) -- (0,0);
	\draw (1.5,2.2) -- (2.5,3.2) -- (4.5,3.2) -- (3.5,2.2) -- (1.5,2.2); 
	\draw (1,1) -- (2,2) -- (4,2) -- (3,1) -- (1,1); 
	\draw [line width=.3mm](0,1) -- (.8,2); 
	\draw [dashed] (0,1) -- (1.5,2.2);
	\draw [dashed] (.8,2) -- (2.5,3.2);
	\draw [dashed] (0,1) -- (1,1);
	\draw [dashed] (.8,2) -- (2,2);
	\node[above] at (3,2.5){$+$};
	\node[above] at (2.5,1.2){$-$};
	\node[right] at (4.5,3.2){$S_1$};
	\node[right] at (4,2){$S_2$};
	\node[left] at (0,3){$H$};
	\node[left] at (0,1){$T$};
	\node[left] at (.7,2.5){$\mathcal K$}; 
	\draw[rotate=5] (.6,2.4) ellipse (7.5pt and 15pt);
\end{tikzpicture}
\caption{The hyperplanes giving a capacitor word. }
\end{figure}

Now we turn to the question of the minumum weight of $\mathcal D$. As mentioned
in the Introduction, we have a general lower bound of $\abs{\K}+1$ for the minimum
weight, but this bound may be improved for certain fields and subsets $\K$.
We record in the next lemma a few facts about the code $\mathcal D_\mathcal K$ that will be useful.  

\begin{lemma}\label{dfacts}
Let $F$ be any field which may or may not have $q=0$. 
	
\begin{enumerate}
	\item If $\mathcal K^\prime \subseteq \mathcal K$ then $\mathcal D_{\mathcal K} \subseteq \mathcal D_{\K^\prime}$, and hence $d(\mathcal D_{\mathcal K}) \geq d(\mathcal D_{\mathcal K^\prime})$. 
	\item If $w$ is a codeword of $\mathcal D_{\K}$, then any line $\ell$ in $L_\K$ is either skew to the support of $w$, or contains at least two points of the support of $w$. 
	\item $d(\mathcal D_{\mathcal K}) \geq |\mathcal K|+1$.
	\item If $w$ is a codeword of $\mathcal D_{\mathcal K}$ for $\mathcal K$ properly contained in a line, then either $wt(w) \geq 2|\mathcal K|+2$ or the support of $w$ is contained in an affine plane whose line at infinity contains $\mathcal K$. 
\end{enumerate}
	
\end{lemma}

\begin{proof}
	
	For (1), if a word satisfies the parity-check equations of $N^T_\mathcal K$ then it also satisfies the equations of $N^T_{\mathcal K^\prime}$ for any $\mathcal K^\prime \subseteq \mathcal K$. 
	
	For (2), if a line $\ell \in L_\mathcal K$ meets exactly one point of the support of $w$, then $w$ does not satisfy the parity-check equation associated to $\ell$ and hence cannot be a codeword. 
	
	Part (3) was proven in the Introduction, so it only remains to prove part (4). Suppose there are two affine planes $\pi_1, \pi_2$ with line at infinity containing $\mathcal K$ and both containing a point in the support of $w$. Using the same argument as the proof of part (3), each plane $\pi_1, \pi_2$ must contain at least $|\mathcal K|+1$ distinct points so that $w$ has weight at least $2|\mathcal K|+2$.   	
\end{proof}

\begin{remark}
	
	Part (4) of Lemma~\ref{dfacts} implies that when $\mathcal K$ is properly contained in a line, either $d(\mathcal D_{\mathcal K}) \geq 2|\mathcal K|+2$ or $\mathcal D_{\mathcal K}$ has the same minimum distance as the code given by the parity-check matrix $N^T_{\mathcal K}$ restricted to the lines of $T_1^*(\mathcal K)$. 
	 	
\end{remark}

\begin{theorem}\label{Kline}
	
	Let $F$ be a totally ordered field. Then the minimum distance of the $F$-code $\mathcal D_{\mathcal K}$ is at least $2|\mathcal K|$. In particular, when $F= \Q$ or $\R$ then $d(\mathcal D_{\mathcal K})\geq 2|\mathcal K|$.
	
\end{theorem}

\begin{proof}
		
Let $w$ be a codeword in $\mathcal D_{\mathcal K}$, and let $s$ be a point in the support of $w$, where without loss of generality we assume that $w(s)>0$. There are $|\mathcal K|$ lines containing $s$, each of which contains a point $t$ in the support of $w$ with $w(t)<0$. Each such point $t$ is also met by $|\mathcal K|$ lines, each of which contains a point $u$ in the support of $w$ with $w(u)>0$. Therefore the sets
\[ \{x \in P : w(x) > 0  \}   \]
\[ \{x \in P : w(x) < 0  \}    \] 
both have cardinality at least $|\mathcal K|$. 	
\end{proof}

\begin{figure}
	\begin{tikzpicture}[scale=1]
		\draw (0,0) -- (0,1) -- (1,2) -- (2,2) -- (2,1) -- (1,0) -- (0,0);
		\draw[fill] (0,0) circle [radius=0.1];
		\draw[fill] (0,1) circle [radius=0.1];
		\draw[fill] (2,1) circle [radius=0.1];
		\draw[fill] (2,2) circle [radius=0.1];
		\draw[fill] (1,2) circle [radius=0.1];
		\draw[fill] (1,0) circle [radius=0.1];
		\node[left] at (0,1) {$+1$};
		\node[above] at (1,2) {$-1$};
		\node[above] at (2,2) {$+1$};
		\node[right] at (2,1) {$-1$};
		\node[below] at (1,0) {$+1$};
		\node[below] at (0,0) {$-1$};
		\draw (6,0) -- (5,1) -- (5,2) -- (6,3) -- (7,3) -- (8,2) -- (8,1) -- (7,0) -- (6,0);
		\draw[fill] (6,0) circle [radius=0.1];
		\draw[fill] (5,1) circle [radius=0.1];
		\draw[fill] (5,2) circle [radius=0.1];
		\draw[fill] (6,3) circle [radius=0.1];
		\draw[fill] (7,3) circle [radius=0.1];
		\draw[fill] (8,2) circle [radius=0.1];
		\draw[fill] (8,1) circle [radius=0.1];
		\draw[fill] (7,0) circle [radius=0.1];
		\node[below] at (6,0) {$+1$};
		\node[left] at (5,1) {$-1$};
		\node[left] at (5,2) {$+1$};
		\node[above] at (6,3) {$-1$};
		\node[above] at (7,3) {$+1$};
		\node[right] at (8,2) {$-1$};
		\node[right] at (8,1) {$+1$};
		\node[below] at (7,0) {$-1$};
	\end{tikzpicture}
\caption{The codewords $w$ (left) and $w^\prime$ (right) of weights $6$ and $8$, respectively, given in Example~\ref{kgons}.}
\end{figure}
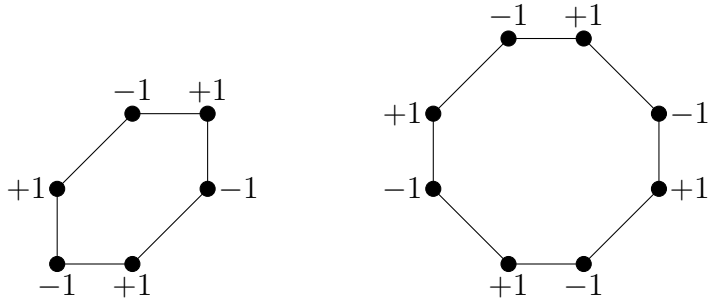

\begin{example}
	Let $\ell$ be a line of $\PG(n,q)$ contained in $H$, and let $v$ be a point on $\ell$. Let $\mathcal K$ be a subset of $\ell \setminus \{v\}$. If $\ell_1$ and $\ell_2$ are any two affine lines with point at infinity $v$, and both lying in an affine plane with line at infinity $\ell$, then the word $\chi_{\ell_1}-\chi_{\ell_2}$ is a codeword of $\mathcal D_{\mathcal K}$ of weight $2q$.		
\end{example}

\begin{example}\label{kgons}
	
  When $\mathcal K$ is a particular set of $3$ or $4$ points contained in a line we can give explicit codewords of weight $6$ and $8$, respectively. As proved in Lemma~\ref{dfacts} these words are contained in an affine plane.
  In these examples, we assume that the characteristic of $\Fq$
  is large enough for the given points to be distinct. Let \[ \mathcal K = \{ \langle (1,0) \rangle, \langle (0,1) \rangle, \langle (1,1) \rangle \}, \] \[ \mathcal K^\prime = \{ \langle (1,0) \rangle, \langle (0,1) \rangle   ,\langle (1,-1) \rangle , \langle (1,1) \rangle \}, \]  Then set $w= \sum_{i=1}^6 (-1)^{i} \delta_{a_i}$,  $w^\prime= \sum_{i=1}^8 (-1)^{i} \delta_{b_i}$, where the points $a_i, b_i$ are given in Table 1. It is an easy exercise to check that $w \in \mathcal D_\mathcal K$ and $w^\prime \in \mathcal D_{\mathcal K^\prime}$, and that $\textnormal{wt}(w)=6$, $\textnormal{wt}(w^\prime)=8$. It would be interesting if there were a general construction for similar codewords of $\mathcal D_\mathcal K$ for arbitrary $\mathcal K$. 
	   
\begin{table}
\label{tab}
\caption{}
	\begin{tabular}{l|c|r}
		$i$ & $a_i$ & $b_i$ \\ 
		\hline 
		1 & (0,0) & (0,0) \\
		2 & (0,1) & (0,1) \\
		3 & (1,2) & (1,2) \\
		4 & (2,2) & (2,2) \\
		5 & (2,1) & (3,1) \\
		6 & (1,0) & (3,0) \\
		7 & ---  & (2,-1) \\
		8 & --- & (1,-1) \\
	\end{tabular}
\end{table}

\end{example}

The next examples shows that the minimum distance of $\mathcal D_\mathcal K$ can be less than $2|\mathcal K|$ when $q=0$ in $F$. 

\begin{example} Let $\ell$ be a line of $\PG(n,q)$ contained in $H$ and $\mathcal K=\ell$. Choose a point $v\in\mathcal K$. If $\ell_1$ and $\ell_2$ are any two affine lines with $v$ as point at infinity,  and both lying in a plane with $\ell$ as line at infinity then, in the special case that $q=0$,  the word $\chi_{\ell_1}-\chi_{\ell_2}$ is a codeword of $\mathcal D_{\mathcal K}$ of weight $2q=2\abs{\mathcal K}-2$.
\end{example}

\begin{example}
  In $\PG(n,2)$, let $\mathcal K=H$ and let $w$ be the sum of the characteristic functions of the points of $P=\PG(n,2)\setminus H$. That is, the support of $w$ is the set of points of the affine space $\textnormal{AG}(n,2)$. Every affine line has two points, so when $F=\mathbb{F}_2$ then $w$ is a codeword in $\mathcal D_\mathcal K$ of weight $\abs{\mathcal K}+1$, meeting the bound of \cite{KLF}. Of course, in this case $\mathcal D_\mathcal K=\{0,w\}$ is not a very interesting code.
\end{example}

\section{Applications}\label{app}
In this section we retain the general hypotheses of Theorems~\ref{main} and \ref{codes} and consider the implications of these theorems in some special cases. 

\begin{corollary}\label{fullrank} Assume $q\neq 0$ in $F$. If $\K$ contains a line
  of $H$ then $N$ has full rank $q^n$. 
\end{corollary}
\begin{proof} This is immediate since every hyperplane of $H$ must meet $\K$.
\end{proof}

\begin{remark} In the special case $\K=\PG(n-1,q)$ of this corollary,
$T_{n-1}^*(\K)$ is the 2-design of points and lines in affine space.
\end{remark}

\begin{remark} A {\it blocking set} in $\PG(2,q)$ is any set  of points that
meets every line. (Sets containing a line are considered to be trivial examples of blocking sets.) Clearly, by Theorem~\ref{main}, the set $\K$ is a blocking set
if and only if $N$ has full rank $q^n$. Baer subplanes are nontrivial examples.  
\end{remark}

\subsection{Wenger graphs}
Let $\K=\{(0:1:u:u^2:\cdots:u^{n-1})\mid u\in\Fq\}$. Then the bipartite graph having
$P$ and $L$ as the bipartition of its vertex set, with adjacency defined by point-line incidence, is called
the Wenger graph $W_{n-1}(q)$. Thus the matrix 
$$
A=\begin{pmatrix}0& N\\N^T&0
\end{pmatrix}
$$
is an adjacency matrix of $W_{n-1}(q)$.  This graph has many alternative descriptions. (See \cite{CLL}.)

From Theorem~\ref{main} we can derive the following formula for $\rank_FN$.

\begin{corollary}\label{wenger} Let $F$ be a field in which $q\neq 0$. Then
$\rank_FN$ is equal to the number of polynomials in $\Fq[X]$ of degree $\le n-1$
having a root in $\Fq$.
\end{corollary}
\begin{proof} As $H=\PG(V)$,  a point of $\K$ is a one-dimensional subspace of $V$
of the form $\langle (0,1,u,u^2,\ldots,u^{n-1})\rangle$.
A hyperplane of $H$ is defined by a nonzero linear function on $V$, which we can write in dual coordinates as $\theta=[c_1:c_2:\cdots:c_n]$. The point lies on the hyperplane
if and only if $0=\theta((0,1,u,u^2,\ldots,u^{n-1}))=c_1+c_2u+\cdots+c_nu^{n-1}$.
Thus, the hyperplane meets $\K$ if and only if the polynomial
$c_1+c_2X+\cdots+c_nX^{n-1}$ has a root in $\Fq$. Since a nonzero scalar multiple
of a linear function defines the same hyperplane and a nonzero scalar multiple
of a polynomial has the same roots, the corollary now follows from Theorem~\ref{main}.
\end{proof}

\begin{remark}
It is an easy exercise to count polynomials of degree at most $ n-1$ having no root in $\Fq$. The number of them is
\begin{equation}\label{formula}
(q-1)\sum_{d=0}^{n-1}\sum_{k=0}^d(-1)^k\binom{q}{k}q^{d-k}.
\end{equation}
(cf. \cite[Lemma 2.2]{CLL}.)  It then follows that 
$$
\rank_FN = q^n - (q-1)\sum_{d=0}^{n-1}\sum_{k=0}^d(-1)^k\binom{q}{k}q^{d-k}
$$
\end{remark}
\begin{remark}
As pointed out in \cite{CRV}, the incidence system $T_{n-1}^*(\K)$ is dual, 
in the sense of interchanging the roles of points and lines, to the system (actually several isomorphic systems) described in \cite{CLL}. Of course, dual systems give rise to 
isomorphic bipartite graphs, so \cite{CRV} and \cite{CLL} are studying the same bipartite graphs.
\end{remark}
\begin{remark}
In \cite{CLL}, a proposed open problem was to determine the parameters
of the linear codes whose Tanner graphs are the Wenger graphs. These correspond to the code we have called $\CC$ and its dual. The minimum weight of $\CC$ is $2q$, by Theorem~\ref{codes}. Our corollary shows that \eqref{formula} gives the dimension of such a code over a field where  $q\neq 0$, as the code is defined as  the nullspace of $N$ (or $N^T$).  As mentioned in the Introduction, the dimension could also be deduced by combining Theorem~\ref{main} with the known multiplicity from \cite{CLL} of the eigenvalue zero. This is because Theorem~\ref{main}
shows that the rank is the same for all fields where $q\neq 0$ and is
in particular equal to the rank in characteristic zero, which, for a symmetric
real matrix, is the matrix size  minus the algebraic multiplicity of zero.
The problem of computing $\rank_FN$ has also been considered in unpublished work of M. Tait and C. Timmons, where the formula for $\rank_FN$ in the case $n=4$ was correctly conjectured. 
\end{remark}
\begin{remark}
If $n=3$, the Wenger graphs coincide with the graphs in \cite{LU}
and the codes having these as their Tanner graphs
are denoted $\mathrm{LU}(3,q)$ and $\mathrm{LU}(3,q)^D$. 
They were  studied in \cite{KPPPF} where a conjecture for the dimension of the binary code was made. The conjecture was proved in \cite{SX} and the result generalized to other fields in \cite{V}, where also the connection to $T_2^*(\K)$ was made.
\end{remark}

\subsection{Hyperovals}
Assume $n=3$, and let $q$ be a power of $2$ and $\K$ a hyperoval in $H\cong\PG(2,q)$.
Then it is well known  that $T_2^*(\K)$ is a generalized  quadrangle of order $(q-1,q+1)$ (\cite[3.1.3]{FGQ}).
By definition of a hyperoval, each line of $H$ meets $\K$ in $0$
or $2$ points, and so there are $\binom{q+2}{2}$ lines that meet $\K$.
Therefore,
$$
\rank_FN=1+(q-1)\binom{q+2}{2}
$$ 
for any field $F$ that is not of characteristic $2$. This rank formula does not depend on the choice of hyperoval $\K$.
Note that if we now modify $\K$ by removing any point, then the rank will be unchanged, since certain secant lines through the point just become tangent lines, and the
same lines meet $\K$.   
These ranks were previously computed in \cite{V2} by a different method.

Finally, we could drop our assumption of coprime characteristics and
consider the $\mathbb F_2$-rank of $N$ for hyperovals. This seems to be an interesting
and difficult problem with possible applications in coding theory. 
Some computer results are tabulated in \cite{V2}, and the $\mathbb F_2$-ranks will depend on the choice of hyperoval $\K$.


\begin{thebibliography}{99}

\bibitem{CRV} P. Cara, S. Rottey, G. Van de Voorde,  \emph{A construction for infinite families of semisymmetric graphs revealing their full automorphism group},
J.  Alg. Combinatorics \textbf{39} (2014), 967-–988.

\bibitem{CRV2} P. Cara, S. Rottey, G. Van de Voorde,  \emph{The isomorphism problem for linear representations and their graphs},
Advances in Geometry, \textbf{14} (2014),  353-–367.

\bibitem{CLL} S. Cioaba, F. Lazebnik, Weiqiang Li, \emph{On the spectrum of Wenger graphs}, J. Combinatorial Theory B, \textbf{107} (2014), 132--139.

\bibitem{dCvM} F. De Clerck, H. Van Maldeghem, \emph{On linear representations of $(\alpha,\beta)$-geometrie}, Eur. J. Comb. 15, 3–11 (1994)

\bibitem{RGG} R. G. Gallager, \emph{Low-density parity-check codes}, IRE Trans. Inform. Theory, vol. IT-18, pp.21-28, Jan. 1962 

\bibitem{KPPPF}  J. L. Kim, U. Peled, I. Perepelitsa, V. Pless and S. Friedland, 
\emph{Explicit construction of families of LDPC codes with no 4-cycles}, IEEE Trans. Inform. Theory, \textbf{50} (2004), 2378-–2388.

\bibitem{LU} F. Lazebnik and V. Ustimenko, \emph{Explicit construction of graphs with arbitrary large girth and of large size}, Discrete Appl. Math. \textbf{60} (1997), 275-–284.

\bibitem{FGQ} S. E. Payne, J. A. Thas, \emph{Finite Generalized Quadrangles}, (1985),
Pitman, New York. 

\bibitem{PSV} V. Pepe, L. Storme, G. Van de Voorde,
\emph{Small weight codewords in the LDPC codes arising from linear representations
of geometries}, J. Combin. Des. \textbf{17} (1) (2009), 1-–24.

\bibitem{SX} P.~Sin,  Q.~Xiang, \emph{On the dimension of certain LDPC codes based on
  $q$-regular bipartite graphs},  IEEE Trans. Inform. Theory \textbf{52} no.~8, (2006), 3735--3737.
  
 \bibitem{TXK} H. Tang, J. Xu, Y. Kou, S. Lin, K. Abdel-Ghaffar, \emph{On Algebraic Construction of Gallager and Circulant Low-Density Parity-Check Codes}, IEEE Trans. Inform. Theory \textbf{50} no. 6, (2004), 1269-1279.


\bibitem{V2} P. Vandendriessche, \emph{LDPC codes associated with linear representations of geometries},  Advances in Mathematics of Communications \textbf{4} (3) (2010), 405--417.

\bibitem{V} P. Vandendriessche, \emph{Some low-density parity-check codes derived from finite geometries}, Designs, Codes and Cryptography  \textbf{54} (3) (2010), 287--297.

\bibitem{KLF} Y. Kou, S. Lin, and M. P.C. Fossorier, \emph{Low-Density Parity-Check Codes Based on Finite Geometries: A Rediscovery and New Results}, IEEE Trans. Inform. Theory \textbf{47} no. 7, (2001), 2711-2736  


\end{thebibliography}
\end{document}